\documentclass[reqno]{amsart}

\usepackage{amsmath}
\usepackage{amssymb}

\newtheorem{Theorem}{Theorem}[section]
\newtheorem{Lemma}[Theorem]{Lemma}
\newtheorem{Proposition}[Theorem]{Proposition}

\theoremstyle{definition}
\newtheorem{Definition}[Theorem]{Definition}
\newtheorem{Remark}[Theorem]{Remark}

\newcommand{\thref}[1]{Theorem \ref{#1}}
\newcommand{\leref}[1]{Lemma \ref{#1}}
\newcommand{\prref}[1]{Proposition \ref{#1}}

\newcommand{\seref}[1]{Section \ref{#1}}

\numberwithin{equation}{section}

\newcommand{\cAz}{{\mathcal A}_{z}}
\newcommand{\cAn}{{\mathcal A}_{n}}
\newcommand{\cPz}{{\mathcal P}_{z^{\pm 1}}}
\newcommand{\cPx}{{\mathcal P}_x}

\newcommand{\C}{\mathbb C}
\newcommand{\Z}{\mathbb Z}
\newcommand{\R}{\mathbb R}
\newcommand{\N}{\mathbb N}
\newcommand{\T}{\mathbb T}

\newcommand{\Id}{\mathrm{Id}}
\newcommand{\ff}{\mathfrak{f}}

\newcommand{\al}{{\alpha}}

\newcommand{\ep}{{\varepsilon}}

\newcommand{\alt}{\tilde{\alpha}}
\newcommand{\zt}{\tilde{z}}
\newcommand{\nt}{\tilde{n}}
\newcommand{\Nt}{\tilde{N}}
\newcommand{\xt}{\tilde{x}}

\newcommand{\cDza}{{\mathcal D}_z^{\alpha}}
\newcommand{\cDna}{{\mathcal D}_n^{\alpha}}

\newcommand{\Ph}{\hat{P}}

\newcommand{\bi}{{\mathfrak{b}}}

\newcommand{\E}[1]{E_{q,\,#1}}
\newcommand{\DD}[1]{D_{q,\,#1}}
\newcommand{\fs}[1]{{\triangle}_{q,\,#1}}
\newcommand{\bs}[1]{{\nabla}_{q,\,#1}}

\newcommand{\cL}{{\mathcal L}}
\newcommand{\cDz}{{\mathcal D}_z}

\newcommand{\fLz}{{\mathfrak L}^{z}}
\newcommand{\fLn}{{\mathfrak L}^{n}}

\newcommand{\fpt}[7]{{}_4\phi_3\left[\begin{matrix} #1 , #2, #3, #4 \\
#5, #6, #7 \end{matrix}\,; q,q\right]}

\begin{document}

\title[Bispectral operators for multivariable 
Askey-Wilson polynomials]
{Bispectral commuting difference operators for multivariable 
Askey-Wilson polynomials}

\author[P.~Iliev]{Plamen~Iliev}
\address{School of Mathematics, Georgia Institute of Technology,
Atlanta, GA 30332--0160, USA}
\email{iliev@math.gatech.edu}

\date{August 14, 2009}

\begin{abstract}
We construct a commutative algebra $\cAz$, generated by $d$ algebraically 
independent $q$-difference operators acting on variables $z_1,z_2,\dots,z_d$, 
which is diagonalized by the multivariable Askey-Wilson polynomials 
$P_n(z)$ considered by Gasper and Rahman \cite{GR2}. Iterating Sears' 
${}_4\phi_3$ transformation formula, we show that the polynomials 
$P_n(z)$ possess a certain duality between $z$ and $n$. Analytic continuation 
allows us to obtain another commutative algebra $\cAn$, generated by $d$ 
algebraically independent difference operators acting on the discrete 
variables $n_1,n_2,\dots,n_d$, which is also diagonalized by $P_n(z)$.
This leads to a multivariable $q$-Askey-scheme of bispectral orthogonal 
polynomials which parallels the theory of symmetric functions.
\end{abstract}

\maketitle

\tableofcontents

\section{Introduction}\label{se1}

In \cite{AW} Askey and Wilson introduced a remarkable family 
$\{p_n(x):n\in\N_0\}$ of orthogonal polynomials on the real line 
depending on four parameters that satisfy a 
second-order $q$-difference equation
\begin{equation}\label{1.1}
A(z)\E{z}(p_n(x))+B(z)p_n(x)+C(z)\E{z}^{-1}(p_n(x))=\lambda_np_n(x),
\end{equation}
where $x=\frac{1}{2}(z+\frac{1}{z})$, $\E{z}$ is the $q$-shift operator acing 
on functions of $z$ by $\E{z}^{\pm 1}(g(z))=g(zq^{\pm1})$, $A(z)$, $B(z)$, 
$C(z)$ are independent of $n$, and $\lambda_n$ is independent of $z$. 
The orthogonality implies that the polynomials $p_n(x)$ satisfy 
also a three-term recursion relation
\begin{equation}\label{1.2}
a_np_{n+1}(x)+b_np_n(x)+c_np_{n-1}(x)=xp_n(x),
\end{equation}
where $a_n$, $b_n$ and $c_n$ are independent of $x$ and $c_0=0$.

If we set $L^z=A(z)\E{z}+B(z)\Id+C(z)\E{z}^{-1}$ then the operator $L^z$
acts naturally on the vector space consisting of all polynomials in the 
variable $x$ 
and equation \eqref{1.1} means that the polynomials $p_n(x)$ are eigenvectors 
of this operator with eigenvalues $\lambda_n$. Similarly, let $E_n$ denote the 
shift operator acting on an arbitrary function $f_n=f(n)$ 
of a discrete variable $n$ by $E_n(f_n)=f_{n+1}$ and let 
$L^n=a_nE_n+b_n\Id+c_nE_n^{-1}$. Then 
the operator $L^n$ acts on the vector space of complex-valued functions 
of a discrete variable $n\in\N_0$ and equation \eqref{1.2} means that $p_n(x)$ 
are eigenvectors of the operator $L^n$ with eigenvalue $x$.
Following \cite{DG} we can say that the Askey-Wilson polynomials solve a 
{\em $q$-difference-difference bispectral problem}. Moreover, this 
bispectral property can be used to characterize the Askey-Wilson polynomials 
\cite{GH1,GH2}.

As the title of the paper \cite{AW} states, the Askey-Wilson polynomials 
generalize those of Jacobi. In fact most of the polynomials used in analysis 
and mathematical physics can be obtained as special or limiting cases 
of these polynomials, see for instance \cite{KS}. For that reason, the 
Askey-Wilson polynomials are at the very top of the $q$-Askey scheme. For the 
terminating branch of the Askey scheme, the bispectral property has a 
natural interpretation within the framework of Leonard pairs \cite{T}.

An interesting question arising from the above discussion is whether there 
is a multivariable $q$-Askey scheme and if the polynomials in this scheme
possess the bispectral property. We refer the reader to the book \cite{DX} for 
a very nice account of the general theory of orthogonal polynomials of several 
variables, which also indicates the two possible ways to proceed.

One possible extension is linked to the theory of symmetric 
functions. In this setting, polynomials which are eigenfunctions of 
$q$-difference operators were proposed by Macdonald \cite{M} and 
Koornwinder \cite{K}, see also \cite{vD} for connections with the bispectral 
problem mentioned above and \cite{vDS} for multivariable $q$-Racah 
polynomials. These polynomial systems are associated with root systems and 
the corresponding polynomials are invariant under the action of the Weyl group.

In a different vein, we can look for multivariable extensions, within the 
usual theory of polynomials of several variables, i.e. we consider an inner 
product in $\R^d$ and we apply the Gram-Schmidt process to all monomials, 
respecting the total degree ordering. Within this context, multivariable 
Askey-Wilson polynomials $\{P_d(n;x;\al):n\in\N_0^d,\; x\in\R^d\}$ depending 
on $d+3$ parameters $\al_0,\al_2,\dots,\al_{d+2}$ were discovered by Gasper 
and Rahman \cite{GR2}. These polynomials represent $q$-analogs of the 
multivariable Wilson polynomials defined by Tratnik \cite{Tr}. 

In the present paper we prove that the polynomials of Gasper 
and Rahman are also bispectral. More precisely, we define two commutative 
algebras $\cAz$ and $\cAn$ of operators which are simultaneously diagonalized 
by the polynomials $P_d(n;x;\al)$. The algebra $\cAz$ is generated by 
$d$ algebraically independent  $q$-difference operators $\fLz_1,\dots,\fLz_d$ 
acting on the variables $z_1,z_2,\dots,z_d$ where 
$x_j=\frac{1}{2}(z_j+z_{j}^{-1})$. Thus the operators $\fLz_j$ act naturally 
on the vector space $\C[x_1,\dots,x_d]$ and are diagonalized by 
the polynomials $P_d(n;x;\al)$. 
Likewise, the algebra $\cAn$ is generated by $d$ algebraically independent 
difference operators $\fLn_1,\dots,\fLn_d$ acting on the 
variables $n_1,\dots,n_d$ and the polynomials $P_d(n;x;\al)$ are 
eigenfunctions of these operators (considered as elements of the vector space 
of complex-valued functions defined on $\N_0^d$).

The paper extends the results in our joint work 
with Geronimo \cite{GI} and it raises a series of interesting questions. For 
instance in the one-dimensional case, for specific values of the free 
parameters, the Askey-Wilson recurrence operator $L^n$ commutes with a 
difference operator of odd order. Jointly with Haine \cite{HI2} we explored 
this fact to connect the above theory to algebraic curves, having specific 
singularities, or equivalently, to specific soliton solutions of the Toda 
lattice hierarchy, using the correspondence established by van 
Moerbeke-Mumford \cite{vMM,Mum} and Krichever \cite{Kr}. 
In particular, this led to a different explanation of 
the bispectral property using algebro-geometric considerations \cite{HI1}. 
Moreover, this approach suggested using techniques from integrable systems 
(such as the Darboux transformation) to construct extensions of the 
Askey-Wilson polynomials which satisfy higher-order $q$-difference equations 
\cite{HI2}. In the multivariable case, algebro-geometric methods were used by 
Chalykh \cite{Ch}, within the context of symmetric functions, to give more 
elementary proofs of several of Macdonald's conjectures. It is a challenging 
problem to construct a Baker-Akhiezer type function for the operators 
discussed in this paper and prove the duality using algebro-geometric 
tools. Another interesting question is to classify the multivariable 
orthogonal polynomials that satisfy $q$-difference equations. 
Recently discrete multivariable orthogonal polynomials satisfying second-order 
difference equations were classified in our joint work with Xu \cite{IX}.

The paper is organized as follows. In \seref{se2} we recall the basic 
definitions and orthogonal properties of the one-dimensional Askey-Wilson 
polynomials as well as the multivariable extension proposed by Gasper and 
Rahman. We also show that a change of variables leads to the multivariable 
$q$-Racah polynomials discussed in \cite{GR3}. In \seref{se3} we define a 
$q$-difference operator $\cL_d$ acting on the variables $z_1,z_2,\dots,z_d$. 
We prove that it preserves the ring of polynomials in $x_1,x_2,\dots,x_d$ 
and we establish its triangular structure\footnote{with respect to 
the total degree ordering}. In \seref{se4} we show that 
$\cL_d$ is self-adjoint with respect to the inner product defined 
by the multivariable Askey-Wilson measure. This leads to the construction 
of the commutative algebra $\cAz$ diagonalized by the multivariable 
Askey-Wilson polynomials. In \seref{se5} we iterate an identity of Sears 
to show that, appropriately normalized, the polynomials $P_d(n;x;\al)$ 
possess a certain duality between the discrete variables 
$(n_1,n_2,\dots,n_d)$ and the continuous variables $(z_1,z_2,\dots,z_d)$. 
Using this duality we obtain the commutative algebra $\cAn$ of difference 
operators acting on the variables $n_1,n_2,\dots,n_d$ which is also 
diagonalized by $P_d(n;x;\al)$, thus proving the bispectrality.

\section{Multivariable Askey-Wilson polynomials}\label{se2}
\subsection{Notations and one-dimensional theory}\label{ss2.1}
Throughout the paper we assume that $q$ is a real number in the open interval 
$(0,1)$. We shall use the standard $q$-shifted factorials
\begin{equation*}
(a;q)_n=\prod_{k=0}^{n-1}(1-aq^{k}), \quad
(a;q)_{\infty}=\prod_{k=0}^{\infty}(1-aq^{k}) \text{ and }
(a_1,a_2,\dots,a_k;q)_n=\prod_{j=1}^k(a_j;q)_n,
\end{equation*}
and the corresponding basic hypergeometric series, see \cite{GR1} for more 
details. We denote by $p_n(x;a,b,c,d)$ the Askey-Wilson polynomials \cite{AW}
\begin{equation}\label{2.1}
p_n(x;a,b,c,d)=\frac{(ab,ac,ad;q)_n}{a^n}
\fpt{q^{-n}}{abcdq^{n-1}}{az}{az^{-1}}{ab}{ac}{ad},
\end{equation}
where $x=\frac{1}{2}(z+\frac{1}{z})$. When $a,b,c,d$ are such that 
$\max(|a|,|b|,|c|,|d|)<1$ the polynomials defined by \eqref{2.1} satisfy 
the orthogonality relation
\begin{equation}\label{2.2}
\frac{1}{2\pi }\int_{-1}^{1}p_n(x;a,b,c,d)p_m(x;a,b,c,d)
\frac{w(z;a,b,c,d)}{\sqrt{1-x^2}}dx=\delta_{n,m}h_n,
\end{equation}
where 
\begin{equation}\label{2.3}
w(z;a,b,c,d)=\frac{(z^2,z^{-2};q)_{\infty}}
{(az,az^{-1},bz,bz^{-1},cz,cz^{-1},dz,dz^{-1};q)_{\infty}}
\end{equation}
and
\begin{equation}\label{2.4}
h_n=\frac{(abcdq^{n-1};q)_n(abcdq^{2n};q)_{\infty}}
{(q^{n+1},abq^n,acq^n,adq^n,bcq^n,bdq^n,cdq^n;q)_{\infty}}.
\end{equation}

\subsection{Multivariable extensions}
Consider $x=(x_1,x_2,\dots,x_d)\in\R^d$. For $|x_j|\leq 1$ we put 
\begin{equation}\label{2.5}
x_j=\cos(\theta_j)=\frac{1}{2}(z_j+z_j^{-1}), 
\end{equation}
where $z_j=e^{i\theta_j}$ for $j=1,2,\dots, d$. We define a weight 
function $w(z)=w_d(z;\al)$ depending on $d+3$ nonzero real parameters 
$\al_0,\al_1,\dots,\al_{d+2}$ by
\begin{equation}\label{2.6}
w(z)=\frac{\prod_{j=1}^{d}(z_j^2,z_j^{-2};q)_{\infty}}
{\prod_{k=0}^{d}\prod_{\ep_1,\ep_2\in\{-1,1\}}\;
(\al_{k+1}\al_{k}^{-1}z_{k+1}^{\ep_1}z_{k}^{\ep_2};q)_{\infty}},
\end{equation}
with the convention that $z_0=\al_0$ and $z_{d+1}=\al_{d+2}$. We shall 
assume that the parameters $\al_k$ are such that if $|z_j|=1$ for 
$j=1,2,\dots,d$ then 
$$|\al_{k+1}\al_{k}^{-1}z_{k+1}^{\pm 1}z_{k}^{\pm 1}|<1 \quad\text{ for }\quad
k=0,1,\dots,d.$$
It is easy to see that this is equivalent to the constraints
\begin{equation}\label{2.7}
\begin{split}
&0<|\al_{d+1}|<|\al_{d}|<\cdots<|\al_{1}|<\min(1,|\al_0|^2)\\
&\qquad \frac{|\al_{d+1}|}{|\al_{d}|}<|\al_{d+2}|
<\frac{|\al_{d}|}{|\al_{d+1}|}.
\end{split}
\end{equation}
Let us consider the following measure on $[-1,1]^d$
\begin{subequations}\label{2.8}
\begin{equation}\label{2.8a}
d\mu(x)
=\frac{w(z)}{(2\pi)^d}\prod_{j=1}^d\frac{dx_j}{\sqrt{1-x_j^2}}
\end{equation}
and the corresponding inner product
\begin{equation}\label{2.8b}
\langle f, g \rangle = \int_{[-1,1]^d}f(x)g(x)d\mu(x).
\end{equation}
\end{subequations}

The following theorem established in \cite{GR2} constructs 
an explicit basis of polynomials mutually orthogonal with respect to 
the inner product defined by \eqref{2.6}-\eqref{2.8}.

\begin{Theorem}\label{th2.1}
For $n=(n_1,n_2,\dots,n_d)\in\N_0^d$ define  
\begin{equation}\label{2.9}
P_d(n;x;\al)=\prod_{j=1}^{d}
p_{n_j}\left(x_j;\al_jq^{N_{j-1}},\frac{\al_j}{\al_0^2}
q^{N_{j-1}},
\frac{\al_{j+1}}{\al_{j}}z_{j+1},\frac{\al_{j+1}}{\al_{j}}z_{j+1}^{-1}\right),
\end{equation}
where $N_k=n_1+n_2+\cdots+n_k$ and $N_0=0$. If the parameters 
$\{\al_j\}_{j=0}^{d+2}$ satisfy \eqref{2.7}, then for every $n,m\in\N_0^d$ we 
have
\begin{equation}\label{2.10}
\langle
P_d(n;x;\al), P_d(m;x;\al)\rangle =\delta_{n,m}H_n,
\end{equation}
where the inner product is defined by \eqref{2.6} and \eqref{2.8}, and 
\begin{equation}\label{2.11}
\begin{split}
H_n=&\prod_{k=1}^d
\frac{\left(\frac{\al_{k+1}^2}{\al_0^2}q^{N_{k-1}+N_{k}-1};q\right)_{n_k}
\left(\frac{\al_{k+1}^2}{\al_0^2}q^{2N_{k}};q\right)_{\infty}}
{\left(q^{n_{k}+1},\frac{\al_{k}^2}{\al_{0}^2}q^{N_{k-1}+N_{k}}, 
\frac{\al_{k+1}^2}{\al_{k}^2}q^{n_{k}};q\right)_{\infty}}\\
&\qquad\times 
\prod_{\ep\in\{-1,1\}}\frac{1}
{(\al_{d+1}\al_{d+2}^{\ep}q^{N_d},\al_{d+1}\al_{d+2}^{\ep}\al_0^{-2}q^{N_d}
;q)_{\infty}}.
\end{split}
\end{equation}
\end{Theorem}
For the convenience of the reader we sketch the proof.
\begin{proof}
We can write the weight \eqref{2.6} as $w(z)=w_1(z)w_2(z)$, where 
\begin{subequations}\label{2.12}
\begin{equation}\label{2.12a}
w_1(z)=\frac{(z_1^2,z_1^{-2};q)_{\infty}}
{\prod_{k=0}^{1}\prod_{\ep_1,\ep_2\in\{-1,1\}}\;
(\al_{k+1}\al_{k}^{-1}z_{k+1}^{\ep_1}z_{k}^{\ep_2};q)_{\infty}},
\end{equation}
is the one-dimensional Askey-Wilson weight for $z_1$, depending also on $z_2$ 
and 
\begin{equation}\label{2.12b}
w_2(z)= \frac{\prod_{j=2}^{d}(z_j^2,z_j^{-2};q)_{\infty}}
{\prod_{k=2}^{d}\prod_{\ep_1,\ep_2\in\{-1,1\}}\;
(\al_{k+1}\al_{k}^{-1}z_{k+1}^{\ep_1}z_{k}^{\ep_2};q)_{\infty}},
\end{equation}
\end{subequations}
is independent of $z_1$. Notice also that in the right-hand side of equation 
\eqref{2.9} only the first polynomial 
$p_{n_1}$ depends on $z_1$. Writing the integral representation of the 
inner-product in \eqref{2.10} we see that the only terms depending on $x_1$ 
are $p_{n_1}$, $p_{m_1}$ and $w_1(z)/\sqrt{1-x_1^2}$. Using \eqref{2.2} we 
obtain
\begin{equation*}
\frac{1}{2\pi }\int_{-1}^{1}p_{n_1}p_{m_1}
\frac{w_1(z)}{\sqrt{1-x_1^2}}dx_1=\delta_{n_1,m_1}h_{n_1}\tilde{w}_1(z_2),
\end{equation*}
where 
\begin{equation*}
h_{n_1}=\frac{\left(\frac{\al_2^2}{\al_0^2}q^{n_1-1};q\right)_{n_1}
\left(\frac{\al_2^2}{\al_0^2}q^{2n_1};q\right)_{\infty}}
{\left(q^{n_1+1},\frac{\al_1^2}{\al_0^2}q^{n_1},
\frac{\al_2^2}{\al_1^2}q^{n_1} ;q\right)_{\infty}},
\end{equation*}
and 
\begin{equation*}
\tilde{w}_1(z_2)=\prod_{\ep\in\{-1,1\}}\frac{1}
{\left(\al_2q^{n_1}z_2^{\ep}
,\frac{\al_2}{\al_0^2}q^{n_1}z_2^{\ep};q\right)_{\infty}}.
\end{equation*}
It easy to see that $\tilde{w}_1(z_2)w_2(z)$ coincides with the weight 
$w_{d-1}(\zt,\alt)$ given in \eqref{2.6} for the variables  
$\zt_1=z_2$,\dots,$\zt_{d-1}=z_d$ and parameters 
$\alt_0=\al_0$, $\alt_1=\al_2q^{n_1}$,  $\alt_2=\al_3q^{n_1}$,\dots, 
$\alt_{d}=\al_{d+1}q^{n_1}$, $\alt_{d+1}=\al_{d+2}$. Moreover, if 
we denote $\nt=(n_2,n_3,\dots,n_d)\in\N_0^{d-1}$ one can check that 
$P_{d-1}(\nt;\xt;\alt)$ coincides with the product $\prod_{j=2}^{d}p_{n_j}$ 
consisting of the last $(d-1)$ polynomials for $P_d(n;x;\al)$. 
The proof now follows by induction. 
\end{proof}

\begin{Remark}\label{re2.2}
One can easily obtain the parametrization used by 
Gasper and Rahman \cite{GR2}, by introducing parameters $a,b,c,d$, 
$a_2,a_3,\dots,a_{d}$ related to $\al_j$ by the formulas
$a=\al_1$, $b=\al_1/\al_0^2$, $c=\al_{d+1}\al_{d+2}/\al_{d}$, 
$d=\al_{d+1}/(\al_{d}\al_{d+2})$ and $a_{k+1}=\al_{k+1}/\al_{k}$ 
for $k=1,2,\dots,d-1$.
\end{Remark}

\begin{Remark}\label{re2.3}
The multivariable $q$-Racah polynomials discussed in \cite{GR3} can be 
obtained as follows. Suppose 
\begin{equation*}
\frac{\al_{d+2}}{\al_{d+1}}=q^{N}, \text{ where }N\in\N.
\end{equation*}
Making the substitution 
\begin{equation*}
z_k=\al_kq^{y_k} \text{ for }k=0,1,\dots,d+1
\end{equation*}
we see that polynomials in \eqref{2.9} become the multivariable 
$q$-Racah polynomials
\begin{equation}\label{2.13}
R_n=\prod_{k=1}^{d}
r_{n_k}\left(y_k-N_{k-1};\frac{\al_k^2}{\al_0^2}q^{2N_{k-1}-1},
\frac{\al_{k+1}^2}{q\al_k^2},\al_k^2q^{y_{k+1}+N_{k-1}},y_{k+1}-N_{k-1}\right)
\end{equation}
where $r_k$ are the one dimensional $q$-Racah polynomials defined by 
\begin{equation*}
r_k(y;a,b,c,N)=(aq,bcq,q^{-N};q)_{k}(q^{N}/c)^{k/2}
\fpt{q^{-k}}{abq^{k+1}}{q^{-y}}{cq^{y-N}}{aq}{bcq}{q^{-N}}.
\end{equation*}
The polynomials $R_n$ are orthogonal on 
$0=y_0\leq y_1\leq y_2\leq \cdots \leq y_d\leq y_{d+1}=N$ 
with respect to the weight
\begin{equation*}
\rho(y)=\prod_{k=0}^d\frac{(\al_{k+1}^2/\al_{k}^2;q)_{y_{k+1}-y_{k}}
(\al_{k+1}^2;q)_{y_{k+1}+y_{k}}}
{(q;q)_{y_{k+1}-y_{k}}(q\al_{k}^2;q)_{y_{k+1}+y_{k}}}
\prod_{k=1}^{d}(1-\al_k^2q^{2y_k})
\left(\frac{\al_{k-1}}{\al_{k}}\right)^{2y_k}.
\end{equation*}
The parametrization in \cite{GR3} can be obtained if we replace 
$\al_0,\dots, \al_{d+1}$ by $a_1,\dots,a_{d+1},b$ related via the 
formulas $a_1=\al_1^2$, $a_k=\al_{k}^2/\al_{k-1}^2$ for $k=2,\dots,d+1$ and 
$b=\al_1^2/(q\al_0^2)$.

Thus, all difference equations derived for the multivariable Askey-Wilson 
polynomials later in the paper, can be translated to the 
multivariable $q$-Racah polynomials, using the change of variables and 
parameters above.
\end{Remark}

\section{$q$-Difference operators in $\C^d$}\label{se3}
In this section we construct a $q$-difference operator $\cL_d$ which is 
triangular and self-adjoint with respect to the inner product 
\eqref{2.6}-\eqref{2.8}.

\subsection{Basic definitions}\label{ss3.1}
Consider the ring $\cPz=\C[z_1^{\pm1},z_2^{\pm1},\dots,z_d^{\pm1}]$ of 
Laurent polynomials in the variables $z_1,z_2,\dots,z_d$ with complex 
coefficients. We denote by $\cPx=\C[x_1,x_2,\dots,x_d]$ the subring of 
$\cPz$ consisting of polynomials in the variables 
$x_1,x_2,\dots,x_d$, where 
\begin{equation}\label{3.1}
x_j=\frac{1}{2}(z_j+z_j^{-1}) \qquad\text{for }j=1,2,\dots,d.
\end{equation}

For $j=1,2,\dots,d$ we define an automorphism $I_j$ of $\cPz$ 
by
\begin{equation}\label{3.2}
I_j(z_j)=z_j^{-1} \text{ and }I_j(z_k)=z_k \text{ for }j\neq k.
\end{equation}
Clearly, $I_j$ is an involution (i.e. $I_j\circ I_j=\Id$) which preserves 
$\cPx$. Conversely, if a polynomial $p\in\cPz$ is preserved 
by the involutions $I_j$ for $j=1,2,\dots,d$ then  $p\in\cPx$.

Let $\{e_1,e_2,\dots,e_d\}$ be the standard basis for $\C^d$. 
We denote by $\E{z_j}$, $\fs{z_j}$ and $\bs{z_j}$, respectively,  
the $q$-shift, forward and backward difference operators in the $j$-th 
coordinate acting on functions $f(z)$ as follows
\begin{align*}
\E{z_j}f(z)&=f(z_1,z_2,\dots,z_jq,\dots,z_d)\\ 
\fs{z_j} f(z)&=(\E{z_j}-1)f(z)\\  
\bs{z_j} f(z)&=(1-\E{z_j}^{-1})f(z).
\end{align*}
Throughout the paper we use the standard multi-index notation. For instance, 
if $\nu=(\nu_1,\nu_2,\dots,\nu_p)\in\Z^d$ then 
$$z^{\nu}=z_1^{\nu_1}z_2^{\nu_2}\cdots z_d^{\nu_d},\quad
\E{z}^{\nu}=\E{z_1}^{\nu_1}\E{z_2}^{\nu_2}\cdots \E{z_d}^{\nu_d},\quad
zq^{\nu}=(z_1q^{\nu_1},z_2q^{\nu_2},\dots,z_dq^{\nu_d})$$
and $|\nu|=\nu_1+\nu_2+\cdots+\nu_d$.

We denote by $\cDz=\C(z_1,\dots,z_d)[\E{z_1}^{\pm1},\dots,\E{z_d}^{\pm1}]$ 
the associative algebra of $q$-difference operators with rational in $z$ 
coefficients, i.e. the elements of $\cDz$ are operators $L$ of the form 
$$L=\sum_{\nu\in S}l_\nu(z)\E{z}^{\nu},$$
where $S$ is a finite subset of $\Z^d$ and $l_{\nu}(z)$ are rational functions 
of $z$. Thus, the algebra $\cDz$ is generated by rational functions of $z$, 
the shift operators $\E{z_1},\E{z_2},\dots,\E{z_d}$ and their inverses 
$\E{z_1}^{-1},\E{z_2}^{-1},\dots,\E{z_d}^{-1}$ subject to the relations
\begin{equation}\label{3.3}
\E{z_j}\cdot g(z)=g(zq^{e_j})\E{z_j},
\end{equation}
for  all rational functions  $g(z)$ and for $j=1,2,\dots,d$. For every 
$k\in\{1,2,\dots,d\}$ the involution $I_k$ can be naturally extended to 
$\cDz$, by defining 
\begin{equation}\label{3.4}
I_k(\E{z_k})=\E{z_k}^{-1} \quad \text{ and }
\quad I_k(\E{z_j})=\E{z_j} \text{ for } j\neq k.
\end{equation} 
Indeed, it is easy to see that $I_k$ is correctly defined because the 
relations \eqref{3.3} are preserved under the action of $I_k$, i.e. we have 
$$I_k(\E{z_j})\cdot I_k(g(z))=I_k(g(zq^{e_j}))I_k(\E{z_j}),$$
for $k,j\in\{1,2,\dots,d\}$.

For the sake of brevity we say that $L\in\cDz$ is $I$-invariant 
if $I_j(L)=L$ for all $j\in\{1,2,\dots, d\}$.

For $k\in\N_0$ we denote by $\cPx^k$ the space of polynomials in 
$\cPx$ of (total) degree at most $k$ in the variables 
$x_1,x_2,\dots,x_d$, with the convention that $\cPx^{-1}=\{0\}$.
\begin{Definition}\label{de3.1}
We say that a linear operator $L$ on $\cPx$ is triangular if for every 
$k\in\N_0$ there is $c_k\in\C$ such that 
$$L(p)=c_k p \mod{\cPx^{k-1}} \text{ for all }p\in\cPx^{k}.$$
\end{Definition}

\subsection{The triangular operator $\cL_d$}\label{ss3.2}
Below we define an operator which is $I$-invariant and triangular. Clearly, an 
operator which is $I$-invariant is uniquely determined by its coefficient of 
$\E{z}^\nu$ (or equivalently $\fs{z}^{\nu}$) with $\nu_i\geq 0$ for 
$i=1,2,\dots,d$.

Let $\nu=(\nu_1,\nu_2,\dots,\nu_d)\in\{0,1\}^d\setminus \{0\}^d$, and let 
$\{\nu_{i_1},\nu_{i_2},\dots,\nu_{i_s}\}$ be the nonzero components of $\nu$ 
with $1\leq i_1<i_2<\cdots<i_s\leq d$ and $s\geq 1$. Denote 
\begin{subequations}\label{3.5}
\begin{equation}\label{3.5a}
\begin{split}
A_{\nu}=&(1-\al_{{i_1}}z_{{i_1}})
(1-\al_{{i_1}}z_{{i_1}}/\al_0^2)\\
&\times \frac{\prod_{k=2}^s
(1-\al_{{i_k}}z_{{i_k}}z_{{i_{k-1}}}/\al_{{i_{k-1}}})
(1-q\al_{{i_k}}z_{{i_k}}z_{{i_{k-1}}}/\al_{{i_{k-1}}})}
{\prod_{k=1}^s(1-z_{{i_k}}^2)(1-qz_{{i_k}}^2)}\\
&\times (1-\al_{d+1}\al_{d+2}z_{{i_s}}/\al_{{i_s}})
(1-\al_{d+1}z_{{i_s}}/(\al_{{i_s}}\al_{d+2})).
\end{split}
\end{equation}

An arbitrary $\nu\in\Z^d$ can be decomposed as $\nu=\nu^{+}-\nu^{-}$, 
where $\nu^{\pm}\in \N_0^d$ with components $\nu_j^{+}=\max(\nu_j,0)$ and 
$\nu_j^{-}=-\min(\nu_j,0)$. For $\nu\in\{-1,0,1\}^d\setminus \{0,1\}^d$ we 
define 
\begin{equation}\label{3.5b}
A_\nu=I^{\nu^{-}}(A_{\nu^{+}+\nu^{-}}).
\end{equation}
\end{subequations}
Here $I^{\nu^{-}}$ is the composition of the involutions corresponding to 
the positive coordinates of $\nu^{-}$.

Finally, we define the operator 
\begin{equation}\label{3.6}
\cL_d=\cL_d(z_1,\dots,z_d;\al_{0},\al_{1},\dots,\al_{d+2})
=\sum_{\nu\in\{-1,0,1\}^d\setminus \{0\}^d}(-1)^{|\nu^{-}|}A_\nu
\fs{z}^{\nu^+}\bs{z}^{\nu^-}.
\end{equation}
Since $I_i(\fs{z_i})=-\bs{z_i}$ and $I_i(\fs{z_j})=\fs{z_j}$ for 
$i\neq j$, the operator defined above is $I$-invariant. 

\begin{Lemma}\label{le3.2} 
Let $L\in\cDz$ be an $I$-invariant $q$-difference operator. If 
\begin{equation}\label{3.7}
\prod_{j=1}^d(1-z_j^2)L(p)\in\cPz \text{ for every }p\in\cPx,
\end{equation}
then $L$ preserves $\cPx$, i.e. $L(\cPx)\subset \cPx$. In particular, 
the operator $\cL_d$ defined by \eqref{3.5}-\eqref{3.6} preserves $\cPx$.
\end{Lemma}

\begin{proof} Let $p(x)\in\cPx$. Since both $L$ and $p$ are $I$-invariant, 
it follows that $L(p)$ is also $I$-invariant and therefore to show that 
$L(p)\in\cPx$ it is enough to prove that $L(p)\in\cPz$. 
From \eqref{3.7} it follows that we can write $L(p)$ as a ratio 
$$L(p)=\frac{p_1(z)}{\prod_{j=1}^d(z_j-z_j^{-1})}, 
\text{ where }p_1(z)\in\cPz.$$
Notice that $I_j(z_j-z_j^{-1})=-(z_j-z_j^{-1})$, and therefore 
$I_j(p_1(z))=-p_1(z)$ for every $j=1,2,\dots,d$. This implies that 
$p_1(z)=0$ when $z_j=\pm 1$ and therefore $p_1(z)$ is divisible 
(in the ring $\cPz$) by $\prod_{j=1}^d(z_j-z_j^{-1})$, proving that 
$L(p)\in\cPz$.

It remains to show that $\cL_d$ satisfies the conditions of the Lemma. 
This follows easily from formulas \eqref{3.5}-\eqref{3.6} combined with the 
fact that 
$\fs{z_j}(x_j^k)$ is divisible by 
$\fs{z_j}(x_j)=\frac{(z_j^2q-1)(q-1)}{2z_jq}$, and 
$\bs{z_j}(x_j^k)$ is divisible by 
$\bs{z_j}(x_j)=\frac{(z_j^2-q)(q-1)}{2z_jq}$ for every $k\in\N$.
\end{proof}

\begin{Proposition}\label{pr3.3}
The operator $\cL_d$ defined by \eqref{3.5}-\eqref{3.6} is triangular. 
More precisely, we have
\begin{equation}\label{3.8}
\cL_d(p(x))
= -(1-q^{-k})\left(1-\frac{\al_{d+1}^2}{\al_0^2}q^{k-1}\right)p(x)
\mod{\cPx^{k-1}} 
\text{ for every } p(x)\in\cPx^k.
\end{equation}
\end{Proposition}

\begin{proof} From \leref{le3.2} we know that $\cL_d(p(x))\in\cPx$ for 
every $p(x)\in\cPx$. It is enough to prove that \eqref{3.8} holds 
when $p(x)=x^n=x_1^{n_1}x_2^{n_2}\cdots x_d^{n_d}$. Clearly the highest 
(total) power of $x$ in $\cL_d(x^n)$ can be uniquely determined by the highest 
power of $z$. Thus it is enough to see that, after canceling the denominators
of the coefficients $A_{\nu}$ of the operator $\cL_d$, we have 
\begin{equation*}
\begin{split}
\cL_d(x^n)&=-\frac{1}{2^{|n|}}(1-q^{-|n|})\left(1-\frac{\al_{d+1}^2}
{\al_0^2}q^{|n|-1}\right)z^n\\
&\qquad +\text{ a linear combination of $z^k$ with $|k|<|n|$}.
\end{split}
\end{equation*}
Let $0\neq\nu\in\{-1,0,1\}^{d}$. Notice that 
\begin{subequations}\label{3.9}
\begin{align}
\fs{z_j}(x_j^{n_j})&= \frac{1}{2^{n_j}}(q^{n_j}-1)z_j^{n_j}+O(z_j^{n_j-1})
                \label{3.9a}\\
\bs{z_j}(x_j^{n_j})&= \frac{1}{2^{n_j}}(1-q^{-n_j})z_j^{n_j}+O(z_j^{n_j-1})
                \label{3.9b}
\end{align}
\end{subequations}
and therefore $\fs{z}^{\nu+}\bs{z}^{\nu_-}(z^n)$ is a linear combination 
of $z^n$ and $z^k$ where $|k|<|n|$. From \eqref{3.5} it follows 
that the denominator of $A_{\nu}$ is a polynomial of degree $4s$ in 
$z_1,z_2,\dots,z_d$, where $s$ denotes the number of nonzero components 
of $\nu$. However, the numerator is of degree $4s$ if and only if all the 
nonzero coordinates of $\nu$ have the same sign (i.e. they are all positive, 
or all negative). If at least two of the coordinates of $\nu$ have 
different signs, then the degree of the numerator is less than $4s$ in 
the variables $z_1,z_2,\dots,z_d$. Thus we conclude that 
$\cL_d(x^n)$ is a linear combination of $z^n$ and $z^k$ where $|k|<|n|$. 
Moreover, the coefficient of $z^n$ can be computed by extracting the 
coefficient of $z^n$ from the terms 
$(-1)^{|\nu^{-}|}A_\nu\fs{z}^{\nu^+}\bs{z}^{\nu^-}(x^n)$ for 
$0\neq \nu\in\{0,1\}^d\cup\{-1,0\}^d$ in the representation \eqref{3.6} of 
the operator $\cL_d$. Using formulas \eqref{3.9} we deduce that 
\begin{equation*}
\cL_d(x^n)=\frac{c_n}{2^{|n|}}z^n+\text{ linear combination of $z^k$ with 
$|k|<|n|$},
\end{equation*}
where $c_n=c_n^{+}+c_n^{-}$ with 
\begin{subequations}\label{3.10}
\begin{align}
c_n^{+}&= \frac{\al_{d+1}^2}{q\al_{0}^2}
\sum_{0\neq\nu\in\{0,1\}^d}
\prod_ {\begin{subarray}{c} s\in\{1,\dots,d\}\\ 
\text{such that }\nu_s=1 \end{subarray}}(q^{n_s}-1)\label{3.10a}\\
c_n^{-}&=\sum_{0\neq\nu\in\{0,-1\}^d}
\prod_ {\begin{subarray}{c} s\in\{1,\dots,d\}\\ 
\text{such that }\nu_s=-1 \end{subarray}}(q^{-n_s}-1).
\label{3.10b}
\end{align}
\end{subequations}
It is easy to see (for instance by induction on $d$) that 
\begin{equation*}
\sum_{0\neq\nu\in\{0,1\}^d}
\prod_ {\begin{subarray}{c} s\in\{1,\dots,d\}\\ 
\text{such that }\nu_s=1 \end{subarray}}(q^{n_s}-1)=q^{|n|}-1,
\end{equation*}
which combined with \eqref{3.10} gives
\begin{equation*}
c_n^{+}=\frac{\al_{d+1}^2}{q\al_{0}^2}(q^{|n|}-1)\quad\text{ and }\quad
c_n^{-}=q^{-|n|}-1.
\end{equation*}
Thus 
\begin{equation*}
c_n=c_n^{+}+c_n^{-}=-(1-q^{-|n|})
\left(1-\frac{\al_{d+1}^2}{\al_0^2}q^{|n|-1}\right),
\end{equation*}
completing the proof.
\end{proof}

\section{The commutative algebra $\cAz$}\label{se4}
In this section we show that the operator $\cL_d$ defined in \seref{se3} 
is self-adjoint with respect to the inner product \eqref{2.6}-\eqref{2.8}. 
This allows us to construct a commutative subalgebra $\cAz$ of $\cDz$, 
generated by $d$ algebraically independent $q$-difference operators, which 
is diagonalized by the polynomials \eqref{2.9}.

\subsection{Self-adjointness of $\cL_d$}\label{ss4.1}
We start with a simple lemma.
\begin{Lemma}\label{le4.1} Consider an inner product defined by equations 
\eqref{2.8} and let $L=\sum_{\nu\in S} C_{\nu}(z)\E{z}^{\nu}\in\cDz$. 
If for every $\nu\in S$ the following conditions are 
satisfied
\begin{itemize}
\item[(i)] $C_{\nu}(z)w(z)=C_{-\nu}(zq^{\nu})w(zq^{\nu})$;
\item[(ii)] The function $C_{\nu}(z)w(z)$ is holomorphic on the 
$d$-dimensional torus\\ 
$\T^{\nu}_t=\{z\in\C^d:|z_j|=q^{-t\nu_j}\text{ for }j=1,2,\dots,d\}$ 
for every $t\in[0,1]$
\end{itemize}
then $L$ is self-adjoint with respect to the inner product \eqref{2.8}.
\end{Lemma}

\begin{proof}
Using the change of variables \eqref{2.5} we can write the 
inner product as an integral over the $d$-dimensional torus 
$\T=\{z\in\C^d:|z_j|=1\text{ for }j=1,2,\dots,d\}$ 
\begin{equation}\label{4.1}
\langle f, g \rangle = \frac{1}{(4\pi i)^d}\int_{\T}f(z)g(z)w(z)
\prod_{j=1}^d\frac{dz_j}{z_j}.
\end{equation}
Denote $L_{\nu}=C_{\nu}(z)\E{z}^{\nu}$. Then from \eqref{4.1} we 
obtain
\begin{align*}
\langle f, L_{-\nu}g\rangle &= \frac{1}{(4\pi i)^d}
\int_{\T}f(z)C_{-\nu}(z)g(zq^{-\nu})w(z)\prod_{j=1}^d\frac{dz_j}{z_j}\\
\intertext{replacing $z$ by $zq^{\nu}$}
&= \frac{1}{(4\pi i)^d}
\int_{\T^{\nu}_{1}}C_{-\nu}(zq^{\nu})f(zq^{\nu})g(z)w(zq^{\nu})
\prod_{j=1}^d\frac{dz_j}{z_j}\\
\intertext{using (i)}
&=\frac{1}{(4\pi i)^d}
\int_{\T^{\nu}_1}C_{\nu}(z)f(zq^{\nu})g(z)w(z)
\prod_{j=1}^d\frac{dz_j}{z_j}\\
\intertext{using (ii) we can replace $\T^{\nu}_1$ by $\T^{\nu}_0=\T$}
&=\frac{1}{(4\pi i)^d}
\int_{\T}C_{\nu}(z)f(zq^{\nu})g(z)w(z)
\prod_{j=1}^d\frac{dz_j}{z_j}\\
&=\langle L_{\nu} f,  g\rangle.
\end{align*}
The proof follows immediately by writing $L$ as a sum of $L_{\nu}$'s.
\end{proof}
In order to apply \leref{le4.1} we need to rewrite the operator $\cL_d$ 
as a linear combination of the $q$-shift operators $\E{z}^{\nu}$. 

For $j\in\{0,1,\dots,d\}$ and $(k,l)\in\{0,1\}^2$ we define $B_{j}^{k,l}(z)$
as follows
\begin{subequations}\label{4.2}
\begin{align}
B_j^{0,0}(z)&=1+\frac{\al_{j+1}^2}{q\al_{j}^2}
-\frac{4\al_{j+1}x_{j}x_{j+1}}{(q+1)\al_{j}}\label{4.2a}\\
B_j^{0,1}(z)&=\left(1-\frac{\al_{j+1}z_{j}z_{j+1}}{\al_{j}}\right)
\left(1-\frac{\al_{j+1}z_{j+1}}{\al_{j}z_{j}}\right)
                                     \label{4.2b}\\
B_j^{1,0}(z)&=\left(1-\frac{\al_{j+1}z_{j}z_{j+1}}{\al_{j}}\right)
\left(1-\frac{\al_{j+1}z_{j}}{\al_{j}z_{j+1}}\right)\label{4.2c}\\
B_j^{1,1}(z)&=\left(1-\frac{\al_{j+1}z_{j}z_{j+1}}{\al_{j}}\right)
\left(1-\frac{q\al_{j+1}z_{j+1}z_{j}}{\al_{j}}\right)
.\label{4.2d}
\end{align}
In the above formulas $x_j$ is related to $z_j$ by \eqref{2.5}, with the 
convention $z_0=\al_0$ and $z_{d+1}=\al_{d+2}$. We extend the definition 
of $B_{j}^{k,l}(z)$ for $(k,l)\in\{-1,0,1\}^2$ by defining 
\begin{align}
B_j^{-1,l}(z)&=I_j(B_j^{1,l}(z)) \text{ for }l=0,1 \label{4.2e}\\
B_j^{k,-1}(z)&=I_{j+1}(B_j^{k,1}(z)) \text{ for }k=0,1 \label{4.2f}\\
B_j^{-1,-1}(z)&=I_{j}(I_{j+1}(B_j^{1,1}(z))).\label{4.2g}
\end{align}
\end{subequations}
Next, for $j\in\{1,\dots,d\}$ we denote
\begin{subequations}\label{4.3}
\begin{align}
b_j^{0}(z)&=(1-qz_{j}^2)(1-qz_{j}^{-2})\label{4.3a}\\
b_j^{1}(z)&=(1-z_{j}^2)(1-qz_{j}^2)\label{4.3b}\\
b_j^{-1}(z)&=I_j(b_j^{1}(z)).\label{4.3c}
\end{align}
\end{subequations}
Finally, for $\nu\in\{-1,0,1\}^d$ we put
\begin{equation}\label{4.4}
C_{\nu}(z)=\left(q(q+1)\right)^{d-|{\nu}^{+}|-|\nu^{-}|}\,
\frac{\prod_{k=0}^{d}B_k^{{\nu}_k,{\nu}_{k+1}}(z)}
{\prod_{k=1}^{d}b_k^{{\nu}_k}(z)},
\end{equation}
with the convention ${\nu}_0={\nu}_{d+1}=0$.

\begin{Proposition}\label{pr4.2}
The operator $\cL_d=\cL_d(z;\al)$ defined by \eqref{3.5}-\eqref{3.6} can be 
written as
\begin{equation}\label{4.5}
\cL_d(z;\al)=\sum_{\nu\in\{-1,0,1\}^d}C_{\nu}(z)\E{z}^{\nu}
-\left(1+\frac{\al_{d+1}^2}{q\al_{0}^2}
-\frac{4\al_{d+1}x_0x_{d+1}}{(q+1)\al_{0}}\right),
\end{equation}
where $C_{\nu}(z)$ are given by \eqref{4.2}, \eqref{4.3} and \eqref{4.4}.
\end{Proposition}

\begin{proof} 
We shall prove the statement by induction on $d$. For $d=1$ 
formulas \eqref{3.5} give
\begin{align*}
&A_{e_1}=\frac{(1-\al_1z_1)\left(1-\frac{\al_1z_1}{\al_0^2}\right)
\left(1-\frac{\al_2\al_3z_1}{\al_1}\right)
\left(1-\frac{\al_2z_1}{\al_1\al_3}\right)}
{(1-z_1^2)(1-qz_1^2)}\\
&A_{-e_1}=I_1(A_{e_1})
=\frac{(z_1-\al_1)\left(z_1-\frac{\al_1}{\al_0^2}\right)
\left(z_1-\frac{\al_2\al_3}{\al_1}\right)
\left(z_1-\frac{\al_2}{\al_1\al_3}\right)}
{(1-z_1^2)(q-z_1^2)}.
\end{align*}
Using \eqref{4.2}, \eqref{4.3} and \eqref{4.4} we get
\begin{align*}
&C_{e_1}=\frac{B_0^{0,1}(z)B_1^{1,0}(z)}{b_1^1(z)}=
\frac{(1-\al_1z_1)\left(1-\frac{\al_1z_1}{\al_0^2}\right)
\left(1-\frac{\al_2\al_3z_1}{\al_1}\right)
\left(1-\frac{\al_2z_1}{\al_1\al_3}\right)}
{(1-z_1^2)(1-qz_1^2)}
=A_{e_1}\\
&C_{-e_1}=I_{1}(C_{e_1}(z))=A_{-e_1} \\
&C_{0}=q(q+1)\frac{B_0^{0,0}(z)B_1^{0,0}(z)}{b_1^0(z)}\\
&\quad\;=q(q+1)
\frac{\left(1+\frac{\al_1^2}{q\al_0^2}-\frac{4\al_1}{(q+1)\al_0}x_0x_1\right)
\left(1+\frac{\al_2^2}{q\al_1^2}-\frac{4\al_2}{(q+1)\al_1}x_1x_2\right)}
{(1-qz_1^2)(1-qz_1^{-2})},
\end{align*}
where $x_0=\frac{1}{2}(\al_0+\al_0^{-1})$, $x_1=\frac{1}{2}(z_1+z_1^{-1})$ 
and $x_2=\frac{1}{2}(\al_{3}+\al_{3}^{-1})$.
We need to check that
\begin{align*}
\cL_1&=A_{e_1}(\E{z_1}-1)-A_{-e_1}(1-\E{z_1}^{-1})\\
&=C_{e_1}\E{z_1}+C_{-e_1}\E{z_1}^{-1}
+C_0-\left(1+\frac{\al_2^2}{q\al_0^2}-\frac{4\al_2x_0x_2}{(q+1)\al_0}\right),
\end{align*}
which amounts to checking that
$$-A_{e_1}-A_{-e_1}
=C_0-\left(1+\frac{\al_2^2}{q\al_0^2}-\frac{4\al_2x_0x_2}{(q+1)\al_0}\right).$$
This equality can be verified by a straightforward computation using 
the explicit formulas for $A_{e_1}$, $A_{-e_1}$ and $C_0$ above.\\

Let now $d>1$ and assume that the statement is true for $d-1$. 
We can write $\cL_d$ as follows
\begin{equation}\label{4.6}
\begin{split}
\cL_{d}
=&\cL'\;\frac{\left(1-\frac{\al_{d+1}\al_{d+2}z_d}{\al_{d}}\right)
\left(1-\frac{\al_{d+1}z_d}{\al_{d}\al_{d+2}}\right)}
{(1-z_d^2)(1-qz_d^2)}\fs{z_d} \\
&+\cL''\;\frac{\left(1-\frac{\al_{d+1}\al_{d+2}}{\al_{d}z_d}\right)
\left(1-\frac{\al_{d+1}}{\al_{d}\al_{d+2}z_d}\right)}
{(1-z_d^{-2})(1-qz_d^{-2})}
(-\bs{z_d})+\cL''',
\end{split}
\end{equation}
where $\cL'$, $\cL''$, $\cL'''$ are $q$-difference operators in the variables 
$z_1,z_2,\dots,z_{d-1}$ with coefficients depending on $z_1,\dots,z_{d}$ and 
the parameters $\al_0,\al_1,\dots,\al_{d+2}$.
Clearly, the operators $\cL'$, $\cL''$, $\cL'''$ are uniquely determined 
from $\cL_d$. This implies that they are $I_1,I_2,\dots,I_{d-1}$ invariant 
and therefore they are characterized by the coefficients of 
$\fs{\bar{z}}^{\bar{\nu}}=\fs{z_1}^{\nu_1}\fs{z_2}^{\nu_2}\cdots
\fs{z_{d-1}}^{\nu_{d-1}}$, where $\bar{z}=(z_1,z_2,\dots,z_{d-1})$ and 
$\bar{\nu}=(\nu_1,\nu_2,\dots,\nu_{d-1})\in\{0,1\}^{d-1}$. The coefficient of 
$\fs{\bar{z}}^{\bar{0}}$ in $\cL'$ is equal to 
$(1-\al_dz_d)(1-\al_dz_d/\al_0^2)$ 
(it comes from the term $A_{e_d}\fs{z_d}$ in $\cL_d$). 
Using equations \eqref{3.5} one can see that the coefficients 
of $\fs{\bar{z}}^{\bar{\nu}}$ for 
$\bar{\nu}\in\{0,1\}^{d-1}\setminus\{\bar{0}\}$
are the same as the coefficients of the operator 
$\cL_{d-1}(z_1,\dots,z_{d-1};\al_0,\dots,\al_{d-1},\al_{d}',\al_{d+1}')$ where 
$\al_{d}'=\al_{d}z_dq^{1/2}$ and $\al_{d+1}'=q^{-1/2}$. Notice that
$$1+\frac{(\al_{d}')^2}{q\al_{0}^2}
-\frac{4\al'_{d}x_0x'_{d}}{(q+1)\al_{0}}
=(1-\al_dz_d)(1-\al_dz_d/\al_0^2).$$
Thus, using the induction hypothesis, we deduce that
\begin{subequations}\label{4.7}
\begin{equation}\label{4.7a}
\begin{split}
\cL'=
&\cL_{d-1}(z_1,\dots,z_{d-1};\al_0,\dots,\al_{d-1},\al_{d}z_dq^{1/2},
q^{-1/2})\\
&\qquad\qquad+(1-\al_dz_d)(1-\al_dz_d/\al_0^2)\\
=&\sum_{\nu\in\{-1,0,1\}^{d-1}}C'_{\nu}\E{\bar{z}}^{\nu},
\end{split}
\end{equation}
where $C'_{\nu}$ are computed from \eqref{4.2}-\eqref{4.4} for the operator 
$\cL_{d-1}$ with parameters given in \eqref{4.7a}. For $\cL''$ we have 
\begin{equation}\label{4.7b}
\cL''=I_d(\cL')=\sum_{\nu\in\{-1,0,1\}^{d-1}}I_d(C'_{\nu})\E{\bar{z}}^{\nu}=
\sum_{\nu\in\{-1,0,1\}^{d-1}}C''_{\nu}\E{\bar{z}}^{\nu}.
\end{equation}
For the last operator $\cL'''$ we obtain
\begin{equation}\label{4.7c}
\begin{split}
\cL'''&=
\cL_{d-1}(z_1,\dots,z_{d-1};\al_0,\dots,\al_{d-1},\al_{d+1},\al_{d+2})\\
&= \sum_{\nu\in\{-1,0,1\}^{d-1}}C'''_{\nu}\E{\bar{z}}^{\nu}
-\left(1+\frac{\al_{d+1}^2}{q\al_{0}^2}
-\frac{4\al_{d+1}x_0x_{d+1}}{(q+1)\al_{0}}\right),
\end{split}
\end{equation}
\end{subequations}
where $C'''_{\nu}$ are computed from \eqref{4.2}-\eqref{4.4} for the operator 
$\cL_{d-1}$ with parameters given in \eqref{4.7c}.\\
Notice that the last term in \eqref{4.7c} gives the last term in the 
right-hand side of \eqref{4.5}. Thus, it remains to show that we get the 
stated formulas \eqref{4.4} for the coefficients 
$C_{\nu}$ in \eqref{4.5}, using the decomposition \eqref{4.6}. 
Since the operator $\cL_{d}$ is $I$-invariant, it is enough to prove 
the formulas for $C_{\nu}$ when $\nu$ has nonnegative coordinates. We have 
two possible cases depending on whether $\nu_d=0$ or $\nu_d=1$.

{\it Case 1: $\nu_d=1$.\/} Write $\nu=(\nu',1)$ with 
$\nu'\in\{0,1\}^{d-1}$. From \eqref{4.6} it is clear that $\E{z}^{\nu}$ 
appears only in the first term on the right-hand side and we have
\begin{equation*}
C_{\nu}=C'_{\nu'}
\frac{\left(1-\frac{\al_{d+1}\al_{d+2}z_d}{\al_d}\right)
\left(1-\frac{\al_{d+1}z_d}{\al_d\al_{d+2}}\right)}
{(1-z_d^2)(1-qz_d^2)}.
\end{equation*}
Notice that the factors $(B_k^{{\nu}_k,{\nu}_{k+1}})'$ in formula
\eqref{4.4} for $C'_{\nu'}$ are the same as the factors 
$B_k^{{\nu}_k,{\nu}_{k+1}}$ in formula \eqref{4.4} for $C_{\nu}$ when 
$k=0,1,\dots,d-2$. Similarly, the factors $(b_k^{\nu_k})'$ in formula
\eqref{4.4} for $C'_{\nu'}$ are the same as the factors 
$b_k^{\nu_k}$ in formula \eqref{4.4} for $C_{\nu}$ when 
$k=1,\dots,d-1$. This combined with the last formula above and 
$$b_d^1=(1-z_d^2)(1-qz_d^2), \qquad 
B_d^{1,0}=\left(1-\frac{\al_{d+1}\al_{d+2}z_d}{\al_d}\right)
\left(1-\frac{\al_{d+1}z_d}{\al_d\al_{d+2}}\right)$$
show that in order to complete the proof we need to check that 
$(B_{d-1}^{{\nu}_{d-1},0})'=B_{d-1}^{{\nu}_{d-1},1}$. 
This can be easily verified from the defining relations \eqref{4.2} by 
considering the two possible cases $\nu_{d-1}=0$ and $\nu_{d-1}=1$.

{\it Case 2: $\nu_d=0$.\/} Let us write again $\nu=(\nu',0)$ with 
$\nu'\in\{0,1\}^{d-1}$. Then from \eqref{4.6} we deduce
\begin{align*}
C_{\nu}=& -C'_{\nu'}\frac{\left(1-\frac{\al_{d+1}\al_{d+2}z_d}{\al_d}\right)
\left(1-\frac{\al_{d+1}z_d}{\al_d\al_{d+2}}\right)}
{(1-z_d^2)(1-qz_d^2)}\\
&\qquad -C''_{\nu'}\frac{\left(1-\frac{\al_{d+1}\al_{d+2}}{\al_dz_d}\right)
\left(1-\frac{\al_{d+1}}{\al_d\al_{d+2}z_d}\right)}
{(1-z_d^{-2})(1-qz_d^{-2})}+C'''_{\nu'}.
\end{align*}
We need to check formula \eqref{4.4} for $C_{\nu}$. Again the factors 
$B_k^{\nu_k,\nu_{k+1}}$ for $k=0,1,\dots,d-2$ and the denominator factors 
$b_k^{\nu_k}$ for $k=1,2,\dots,d-1$ are common for $C_{\nu}$, $C'_{\nu'}$, 
$C''_{\nu'}$ and $C'''_{\nu'}$. Thus we need to verify that
\begin{align*}
&q(q+1)\frac{B_{d-1}^{\nu_{d-1},0}B_d^{0,0}}{(1-qz_d^2)(1-qz_d^{-2})}
= -(B_{d-1}^{\nu_{d-1},0})'
\frac{\left(1-\frac{\al_{d+1}\al_{d+2}z_d}{\al_d}\right)
\left(1-\frac{\al_{d+1}z_d}{\al_d\al_{d+2}}\right)}
{(1-z_d^2)(1-qz_d^2)}\\
&\qquad -(B_{d-1}^{\nu_{d-1},0})''
\frac{\left(1-\frac{\al_{d+1}\al_{d+2}}{\al_dz_d}\right)
\left(1-\frac{\al_{d+1}}{\al_d\al_{d+2}z_d}\right)}
{(1-z_d^{-2})(1-qz_d^{-2})}+(B_{d-1}^{\nu_{d-1},0})'''.
\end{align*}
Using the explicit formulas \eqref{4.2} and considering separately the two 
possible cases $\nu_{d-1}=0$ and $\nu_{d-1}=1$ one can check that the above 
equality holds, thus completing the proof.
\end{proof}

\begin{Proposition}\label{pr4.3}
The operator $\cL_d$ defined by \eqref{3.5}-\eqref{3.6} is self-adjoint 
with respect to the inner product \eqref{2.6}-\eqref{2.8}.
\end{Proposition}

\begin{proof}
We shall use \prref{pr4.2} to check that the conditions in \leref{le4.1} 
are satisfied.
We can represent the weight $w(z)$ in \eqref{2.6} as 
$$w(z)=\frac{\prod_{k=1}^{d}w'_k(z)}{\prod_{k=0}^{d}w''_k(z)},$$
where 
\begin{subequations}\label{4.8}
\begin{align}
w'_k(z)&=(z_k^{2},z_k^{-2};q)_{\infty} \label{4.8a}\\
w''_k(z)&=\left(\frac{\al_{k+1}}{\al_k}z_{k+1}z_{k},
\frac{\al_{k+1}}{\al_k}\frac{z_{k+1}}{z_{k}},
\frac{\al_{k+1}}{\al_k}\frac{z_{k}}{z_{k+1}},
\frac{\al_{k+1}}{\al_k}\frac{1}{z_{k}z_{k+1}};q\right)_{\infty}.\label{4.8b}
\end{align}
\end{subequations}
Using \prref{pr4.2} and formula \eqref{4.4} for the coefficients 
of the operator $\cL_d(z;\al)$ we see that condition (i) of \leref{le4.1} 
will be satisfied if we can show that for every $\nu\in\{-1,0,1\}^{d}$ 
we have 
\begin{subequations}\label{4.9}
\begin{equation}\label{4.9a}
\frac{b_k^{\nu_{k}}(z)}{b_k^{-\nu_{k}}(zq^{\nu})}
=\frac{w'_k(z)}{w'_k(zq^{\nu})}
\end{equation}
and
\begin{equation}\label{4.9b}
\frac{B_k^{\nu_{k},\nu_{k+1}}(z)}{B_k^{-\nu_{k},-\nu_{k+1}}(zq^{\nu})}
=\frac{w''_k(z)}{w''_k(zq^{\nu})}.
\end{equation}
\end{subequations}
It is easy to see that if these formulas are true for some $\nu$ then they 
are also true for $-\nu$. Since $b_k^{\nu_{k}}(z)$ depends only on $z_k$, it 
is enough to check \eqref{4.9a} for $\nu_k=0$ and $\nu_k=1$. The case 
$\nu_k=0$ is trivial because clearly both sides of equation \eqref{4.9a} are 
equal to 1. If $\nu_k=1$, then the right-hand side of \eqref{4.9a} is 
$$\frac{w'_k(z)}{w'_k(zq^{\nu})}=
\frac{(z_k^2,z_k^{-2};q)_{\infty}}{(q^{2}z_k^2,z_k^{-2}q^{-2};q)_{\infty}}
=\frac{(1-z_k^2)(1-qz_k^2)}
{\left(1-\frac{1}{q^2z_k^2}\right)\left(1-\frac{1}{qz_k^2}\right)}.$$
Using \eqref{4.3b}-\eqref{4.3c} we see at once that the last expression equals 
$b_k^{1}(z)/b_k^{-1}(zq^{e_k})$ completing the proof of \eqref{4.9a}.
Similarly, $B_k^{\nu_{k},\nu_{k+1}}(z)$ depends only on $z_{k}$ and $z_{k+1}$. 
Thus it is enough to verify \eqref{4.9b} for 
$(\nu_{k},\nu_{k+1})=\{(1,0),(0,1), (1,1), (1,-1)\}$, which can be done in 
the same manner by a straightforward verification using formulas 
\eqref{4.2}.

Condition (ii) in \leref{le4.1} will follow if we show that for every 
$\nu\in\{-1,0,1\}^{d}$ the functions
\begin{subequations}\label{4.10}
\begin{equation}\label{4.10a}
\frac{w_k'(z)}{b_k^{\nu_{k}}(z)}\quad\text{ for}\quad 
k=1,2,\dots,d
\end{equation}
and 
\begin{equation}\label{4.10b}
\frac{B_k^{\nu_{k},\nu_{k+1}}(z)}{w_k''(z)}\quad\text{ for}\quad 
k=0,1,2,\dots,d
\end{equation}
\end{subequations}
are holomorphic on $\T^{\nu}_{t}$ for $t\in[0,1]$. Notice that 
$$w_k'(z)=(1-z_k^{2})(1-qz_k^{2})(1-z_k^{-2})(1-qz_k^{-2})
(z_k^2q^2,z_k^{-2}q^2;q)_{\infty},$$
which combined with \eqref{4.3} shows that all the factors in 
the denominator in \eqref{4.10a} cancel leading to 
a holomorphic function on $\T^{\nu}_{t}$. Next we show that function in 
\eqref{4.10b} is  holomorphic on $\T^{\nu}_{t}$. From 
\eqref{4.9b} it follows that if this is true for $\nu$, then it will also be 
true for $-\nu$. Thus we can consider only the cases: 
$(\nu_{k},\nu_{k+1})=\{(0,0),(1,0),(0,1), (1,1), (1,-1)\}$. The case 
$(\nu_{k},\nu_{k+1})=(0,0)$ is trivial since equations \eqref{2.7} 
guarantee that $w_k''(z)$ has no zeros on the the $d$-dimensional torus 
$\T=\{z\in\C^d:|z_j|=1\text{ for }j=1,2,\dots,d\}$. Let us consider for 
instance the case $(\nu_{k},\nu_{k+1})=(1,1)$. Then for $z\in \T^{\nu}_{t}$ we 
have $1\leq |z_{k}z_{k+1}|\leq q^{-2}$ and $|z_{k}/z_{k+1}|=1$.
From \eqref{2.7} and \eqref{4.8b} it follows that the only factors in 
$w_k''(z)$ that can vanish during the homotopy are 
$$\left(1-\frac{\al_{k+1}z_{k}z_{k+1}}{\al_{k}}\right)
\left(1-\frac{q\al_{k+1}z_{k+1}z_{k}}{\al_{k}}\right).$$
From \eqref{4.2d} we see that these two factors cancel with $B_k^{1,1}$ 
showing that the function in \eqref{4.10b} is holomorphic on $\T^{\nu}_{t}$. 
The remaining three cases follow along the same lines.
\end{proof}

As an immediate corollary of \prref{pr3.3} and \prref{pr4.3} we obtain 
the following theorem.

\begin{Theorem}\label{th4.4}
Let $\al_0,\al_1,\dots,\al_{d+2}$ be real parameters satisfying \eqref{2.7} 
and let $\{Q(n;x;\al):n\in\N_0^{d}\}$ be a set of orthogonal polynomials 
with respect to the inner product \eqref{2.8}, i.e. $Q(n;x;\al)$ is 
a polynomial of total degree $|n|$, orthogonal to all polynomials 
of degree at most $|n|-1$. Then $Q(n;x;\al)$ is an eigenfunction of the 
operator $\cL_d$ defined by \eqref{3.5}-\eqref{3.6} with eigenvalue
$\mu=-(1-q^{-|n|})\left(1-\frac{\al_{d+1}^2}{\al_0^2}q^{|n|-1}\right)$.
\end{Theorem}

\subsection{Construction of $\cAz$}\label{ss4.2}
Next we focus on the polynomials $P_d(n;x;\al)$ defined by 
\eqref{2.1} and \eqref{2.9}. 

\begin{Proposition}\label{pr4.5}
Let $\al\in(\C^*)^{d+3}$.
For $j\in\{1,2,\dots,d\}$ define 
\begin{subequations}\label{4.11}
\begin{align}
\fLz_j&=\cL_j(z_1,\dots,z_j;\al_0,\dots,\al_{j+1},z_{j+1})\label{4.11a}\\
\mu_j&=-(1-q^{-(n_1+n_2+\cdots+n_j)})\left(1-\frac{\al_{j+1}^2}{\al_0^2}
q^{n_1+n_2+\cdots+n_j-1}\right).\label{4.11b}
\end{align}
\end{subequations}
Then the polynomials $P_d(n;x;\al)$ defined by \eqref{2.1} and \eqref{2.9}
satisfy the spectral equations
\begin{equation}\label{4.12}
\fLz_j P_d(n;x;\al) = \mu_j P_d(n;x;\al)
\end{equation}
for every $j\in\{1,2,\dots, d\}$ and the operators $\fLz_j$ commute with each 
other, i.e. $\cAz=\C[\fLz_1,\fLz_2,\dots,\fLz_d]$ is a commutative 
subalgebra of $\cDz$.
\end{Proposition}

\begin{proof} 
If $\al_j$ are real satisfying \eqref{2.7}, then 
\eqref{4.12} for $j=d$ follows immediately from \thref{th4.4}. Notice that 
for fixed $n\in\N_0^d$ both sides of \eqref{4.12} depend rationally on the 
parameters $\al_j$. Thus when $j=d$ equation \eqref{4.12} holds 
for arbitrary values of the parameters $\al_j$. From \eqref{2.9} 
it clear that for $j<d$ the product of the first $j$ terms on the right-hand 
side is precisely 
$$P_j(n_1,\dots,n_j;x_1,\dots,x_j;\al_0,\dots,\al_{j+1},z_{j+1}),$$ 
while the remaining terms $\prod_{l=j+1}^{d}p_{n_l}$ do not depend on the 
variables $x_1,\dots,x_j$.
This shows that \eqref{4.12} holds also for $j<d$.

It remains to prove that the operators $\fLz_j$ commute with each other. 
Fix $i\neq j\in\{1,2,\dots, d\}$. 
From \eqref{4.12} it follows that
$$[\fLz_i,\fLz_j] P_d(n;x;\al)=0, \text{ for all }n\in\N_0^d.$$
This means that the operator $L=[\fLz_i,\fLz_j]\in\cDz$ vanishes on $\cPx$ and 
we want to show that it is identically equal to zero, i.e. its coefficients 
are identically equal to zero. Equivalently, it is enough to show that  
for fixed $k=(k_1,\dots,k_d)\in\N_0^d$ the operator $L'=\E{z}^{k}L$ (which 
also vanishes on $\cPx$) is identically equal to zero. We can assume that 
the components $k_j$ are large enough, and therefore the operator $L'$ will 
contain only nonnegative powers of $\E{z}$. If we denote
\begin{equation*}
\DD{z_j}=\frac{1}{z_j}(1-\E{z_j})
\end{equation*}
then the operator $L'$ can be uniquely written as 
\begin{equation*}
L'=\sum_{\begin{subarray}{c}\nu\in\N_0^d\\ |\nu|\leq M \end{subarray}}
l_{\nu}(z)\DD{z}^{\nu},
\end{equation*}
where $M\in\N$ is large enough and $l_{\nu}(z)$ rational functions of $z$.
Now we use the fact that $L'(x^n)=0$ for all $n\in\N_0^d$ such that 
$|n|\leq M$. We obtain a homogeneous linear system for $l_{\nu}(z)$ with
determinant 
$$\mathrm{Det}=\det_{\begin{subarray}{c}\nu,n\in\N_0^d\\
|\nu|,|n|\leq M\end{subarray}}(\DD{z}^{\nu}x^n).$$
In the above determinant we order in the same way $\nu$ (the rows) and $n$ 
(the columns) respecting the total degree, i.e. $\nu<\mu$ when 
$|\nu|<|\mu|$. It remains to show that $\mathrm{Det}$ is not identically 
equal to $0$. Notice that 
\begin{equation*}
\DD{z_j}^{\nu_j}x_j^{n_j}=\frac{1}{2^{n_j}}(q^{n_j-\nu_j+1};q)_{\nu_j}
z^{n_j-\nu_j}+O(z^{n_j-\nu_j-1})
\end{equation*}
and therefore if we put $$g_{\nu,n}=\frac{\prod_{j=1}^d
(q^{n_j-\nu_j+1};q)_{\nu_j}}{2^{|n|}}$$
then 
\begin{equation*}
\DD{z}^{\nu}x^n=g_{\nu,n}z^{n-\nu}+
\text{ terms involving  $z^k$ with }|k|<|n|-|\nu|.
\end{equation*}
From this formula it follows that
$$\mathrm{Det}=\det_{\begin{subarray}{c}\nu,n\in\N_0^d\\
|\nu|,|n|\leq M\end{subarray}}(g_{\nu,n})+\text{ a linear combination of 
$z^k$ with }|k|<0.$$
Clearly, $g_{\nu,n}\neq 0$ if and only if $\nu_j\leq n_j$ for all $j$.
This shows that matrix $(g_{\nu,n})$ is upper triangular with nonzero entries 
on the main diagonal, hence $\det_{\nu,n}(g_{\nu,n})\neq0$ completing the 
proof.
\end{proof}

\section{Bispectrality}\label{se5}

\subsection{Duality of the multivariable Askey-Wilson polynomials}\label{ss5.1}

Since $q^{n_j}$ determines $n_j$ uniquely modulo $\frac{2\pi i}{\log(q)}\,\Z$, 
we consider in this section complex variables
$n=(n_1,\dots,n_d)\in\left(\C\mod \frac{2\pi i}{\log(q)}\,\Z\right)^d$,
complex variables $z=(z_1,z_2,\dots,z_d)\in(\C^{*})^d$ with nonzero 
components, and nonzero parameters 
$\al=(\al_0,\dots,\al_{d+2})\in(\C^{*})^{d+3}$.

We define dual variables 
$\nt\in\left(\C\mod \frac{2\pi i}{\log(q)}\,\Z\right)^d$ , 
$\zt\in(\C^{*})^d$ and dual parameters $\alt_j\in(\C^{*})^{d+3}$ by
\begin{subequations}\label{5.1}
\begin{align}
q^{\nt_{j}}&=\frac{\al_{d+1-j}z_{d+1-j}}{\al_{d+2-j}z_{d+2-j}}
&&\text{for }j=1,2,\dots,d \label{5.1a} \\
\zt_{j}&=\frac{\al_{d+2-j}}{\al_0}q^{N_{d+1-j}-1/2} 
&&\text{for }j=1,2,\dots,d \label{5.1b}\\
\alt_0&=\al_0\label{5.1c}\\
\alt_j&=\frac{\al_0\al_{d+1}\al_{d+2}}{\al_{d+2-j}}q^{1/2}&&\text{for }
j=1,2,\dots,d+1\label{5.1d}\\
\alt_{d+2}&=\frac{\al_{1}}{\al_{0}}q^{-1/2},\label{5.1e}
\end{align}
\end{subequations}
where as before we have set $z_{d+1}=\al_{d+2}$ and $N_{k}=n_1+\cdots+n_k$.
In analogy with \eqref{2.5} we put $\xt_j=\frac{1}{2}(\zt_j+\zt_j^{-1})$.

\begin{Lemma}\label{le5.1}
The mapping $\ff:(n,z,\al)\rightarrow(\nt,\zt,\alt)$ given by \eqref{5.1}
defines an involution on 
$\left(\C\mod \frac{2\pi i}{\log(q)}\,\Z\right)^d\times (\C^{*})^d\times 
(\C^{*})^{d+3}$.
\end{Lemma}

\begin{proof}
Using formulas \eqref{5.1} one can easily check that $\ff\circ\ff=\Id$. 
\end{proof}

Iterating a formula of Sears we obtain the following lemma.
\begin{Lemma}\label{le5.2}
If $k\in\N_0$ and $abc=defq^{k-1}$ then 
\begin{equation}\label{5.2}
\begin{split}
&(d,e,f;q)_{k}\,\fpt{q^{-k}}{a}{b}{c}{d}{e}{f}=
c^k(b,aq^{1-k}/e,e/c;q)_{k}\\
&\qquad\times
\fpt{q^{-k}}{q^{1-k}/e}{d/b}{f/b}{q^{1-k}/b}{aq^{1-k}/e}{cq^{1-k}/e}.
\end{split}
\end{equation}
\end{Lemma}

\begin{proof}
Sears' formula (see for instance \cite[page~49, formula (2.10.4)]{GR1}) gives
\begin{equation}\label{5.3}
\begin{split}
&(d,e,f;q)_{k}\,\fpt{q^{-k}}{a}{b}{c}{d}{e}{f}=
a^k(d,e/a,f/a,;q)_{k}\\
&\qquad\times
\fpt{q^{-k}}{a}{d/b}{d/c}{d}{aq^{1-k}/e}{aq^{1-k}/f}.
\end{split}
\end{equation}
Applying the same formula on the right-hand side of \eqref{5.3} with 
$a,b,c,d,e,f$ replaced by $d/b,a,d/c,aq^{1-k}/e,d,aq^{1-k}/f$ respectively 
and using the identity 
$$(s;q)_{k}=(-s)^{k}q^{(k-1)k/2}(q^{1-k}/s;q)_{k}$$
we obtain \eqref{5.2}.
\end{proof}

Let us normalize the multivariable Askey-Wilson polynomials \eqref{2.9} as 
follows
\begin{equation}\label{5.4}
\Ph_d(n;x;\al)=\frac{(\al_{d+1}\al_{d+2})^{|n|}P_d(n;x;\al)}
{(\al_{d+1}\al_{d+2},\al_{d+1}\al_{d+2}/\al_0^2;q)_{|n|}
\prod_{j=1}^d\al_j^{n_j}(\al_{j+1}^2/\al_j^2;q)_{n_j}}.
\end{equation}
The main result in this subsection is the following theorem.

\begin{Theorem}\label{th5.3} 
If $(n,z,\al)$ and $(\nt,\zt,\alt)$ are related by 
\eqref{5.1} and if $n,\nt\in\N_0^d$ then
\begin{equation}\label{5.5}
\Ph_d(n;x;\al)=\Ph_d(\nt;\xt;\alt).
\end{equation}
\end{Theorem}

\begin{proof}
Using \eqref{2.1} and applying formula \eqref{5.2} with $k,a,b,c,d,e,f$ 
replaced by 
$n_j$, $\al_{j+1}^2q^{N_{j-1}+N_{j}-1}/\al_0^2$, $\al_jq^{N_{j-1}}z_{j}$, 
$\al_{j}q^{N_{j-1}}z_{j}^{-1}$, $\al_{j}^2q^{2N_{j-1}}/\al_{0}^2$, 
$\al_{j+1}q^{N_{j-1}}z_{j+1}^{-1}$, $\al_{j+1}q^{N_{j-1}}z_{j+1}$
we can write the polynomial $p_{n_j}$ in \eqref{2.9} as follows

\begin{equation}\label{5.6}
\begin{split}
p_{n_j}&=(z_j)^{-n_{j}}
(\al_jq^{N_{j-1}}z_j,\al_{j+1}q^{N_{j-1}}z_{j+1}/\al_0^2,
\al_{j+1}\al_{j}^{-1}z_{j}z_{j+1}^{-1};q)_{n_j}\\
&\qquad\times 
\fpt{q^{-n_j}}{z_{j+1}q^{1-N_j}\al_{j+1}^{-1}}
{\al_0^{-2}\al_{j}q^{N_{j-1}}z_{j}^{-1}}
{\al_{j+1}\al_{j}^{-1}z_{j}^{-1}z_{j+1}}{q^{1-N_j}\al_{j}^{-1}z_{j}^{-1}}
{\al_{0}^{-2}\al_{j+1}q^{N_{j-1}}z_{j+1}}
{q^{1-n_j}\al_{j}\al_{j+1}^{-1}z_{j}^{-1}z_{j+1}}.
\end{split}
\end{equation}
Using the above formula for $p_{n_j}$ and \eqref{5.1} one can check that the 
${}_{4}\phi_3$ term of $p_{n_j}$ in the variables $(n,z,\al)$
coincides with the ${}_{4}\phi_3$ term of $p_{\nt_{d+1-j}}$ in the variables
$(\nt;\zt;\alt)$ for $j=1,2,\dots,d$.
Thus if we compute the ratio $P_d(n;x;\al)/P_d(\nt;\xt;\alt)$ all 
${}_{4}\phi_3$ terms cancel and we get 
\begin{equation}\label{5.7}
\frac{P_d(n;x;\al)}{P_d(\nt;\xt;\alt)}
=\prod_{j=1}^{d}
\frac{(z_j)^{-n_{j}}(\al_jq^{N_{j-1}}z_j,\al_{j+1}q^{N_{j-1}}z_{j+1}/\al_0^2,
\al_{j+1}\al_{j}^{-1}z_{j}z_{j+1}^{-1};q)_{n_j}}
{(\zt_j)^{-\nt_{j}}(\alt_jq^{\Nt_{j-1}}\zt_j,
\alt_{j+1}q^{\Nt_{j-1}}\zt_{j+1}/\alt_0^2,
\alt_{j+1}\alt_{j}^{-1}\zt_{j}\zt_{j+1}^{-1};q)_{\nt_j}}.
\end{equation}
Using \eqref{5.1} we can eliminate $z_j$ and $\zt_j$ in the right-hand side 
of \eqref{5.7}. From \eqref{5.1d}-\eqref{5.1e} it follows that 
$\al_{d+1}\al_{d+2}=\alt_{d+1}\alt_{d+2}.$ 
Next, notice that $\al_jz_j=\al_{d+1}\al_{d+2}q^{\Nt_{d+1-j}}$. Thus we have
\begin{equation*}
\prod_{j=1}^{d}\frac{z_j^{-n_j}}{\zt_j^{-\nt_j}}=
\prod_{j=1}^{d}\frac{\left(\al_j\al_{d+1}^{-1}\al_{d+2}^{-1}\right)^{n_j}}
{\left(\alt_j\alt_{d+1}^{-1}\alt_{d+2}^{-1}\right)^{\nt_{j}}}
\prod_{j=1}^{d}\frac{q^{-\Nt_{d+1-j}n_j}}{q^{-N_{d+1-j}\nt_j}}.
\end{equation*}
It is easy to see that the second product on the right-hand side of the above 
formula is equal to one, by replacing $n_j=N_{j}-N_{j-1}$ and 
$\nt_j=\Nt_{j}-\Nt_{j-1}$. Therefore
\begin{equation}\label{5.8}
\prod_{j=1}^{d}\frac{z_j^{-n_j}}{\zt_j^{-\nt_j}}=
\prod_{j=1}^{d}\frac{\left(\al_j\al_{d+1}^{-1}\al_{d+2}^{-1}\right)^{n_j}}
{\left(\alt_j\alt_{d+1}^{-1}\alt_{d+2}^{-1}\right)^{\nt_{j}}}.
\end{equation}
Using 
the identity 
$$(aq^{l};q)_{k}=\frac{(a;q)_{l+k}}{(a;q)_{l}}\quad\text{ for }\quad 
k,l\in\N_0,$$
we obtain
\begin{equation*}
\begin{split}
&\prod_{j=1}^d\frac{(\al_jq^{N_{j-1}}z_j;q)_{n_j}}
{(\alt_jq^{\Nt_{j-1}}\zt_j;q)_{\nt_j}}
=\prod_{j=1}^d\frac{(\al_{d+1}\al_{d+2}q^{N_{j-1}+\Nt_{d+1-j}};q)_{n_j}}
{(\al_{d+1}\al_{d+2}q^{\Nt_{j-1}+N_{d+1-j}};q)_{\nt_j}}\\
&\qquad=
\prod_{j=1}^d\frac{(\al_{d+1}\al_{d+2};q)_{N_{j}+\Nt_{d+1-j}}}
{(\al_{d+1}\al_{d+2};q)_{\Nt_{j}+N_{d+1-j}}}
\prod_{j=1}^d\frac{(\al_{d+1}\al_{d+2};q)_{\Nt_{j-1}+N_{d+1-j}}}
{(\al_{d+1}\al_{d+2};q)_{N_{j-1}+\Nt_{d+1-j}}}.
\end{split}
\end{equation*}
It is easy to see now that the first product above is equal to 1, while 
the second simplifies to $\frac{(\al_{d+1}\al_{d+2};q)_{N_{d}}}
{(\al_{d+1}\al_{d+2};q)_{\Nt_{d}}}$. Thus we have
\begin{equation}\label{5.9}
\prod_{j=1}^d\frac{(\al_jq^{N_{j-1}}z_j;q)_{n_j}}
{(\alt_jq^{\Nt_{j-1}}\zt_j;q)_{\nt_j}}
=\frac{(\al_{d+1}\al_{d+2};q)_{|n|}}
{(\alt_{d+1}\alt_{d+2};q)_{|\nt|}}.
\end{equation}
Similar manipulations show that
\begin{equation}\label{5.10}
\begin{split}
&\prod_{j=1}^{d}
\frac{(\al_{j+1}q^{N_{j-1}}z_{j+1}/\al_0^2,
\al_{j+1}\al_{j}^{-1}z_{j}z_{j+1}^{-1};q)_{n_j}}
{(\alt_{j+1}q^{\Nt_{j-1}}\zt_{j+1}/\alt_0^2,
\alt_{j+1}\alt_{j}^{-1}\zt_{j}\zt_{j+1}^{-1};q)_{\nt_j}}\\
&\qquad=
\frac{(\al_{d+1}\al_{d+2}/\al_0^2;q)_{|n|}}
{(\alt_{d+1}\alt_{d+2}/\alt_0^2;q)_{|\nt|}}
\prod_{j=1}^{d}
\frac{(\al_{j+1}^2/\al_{j}^2;q)_{n_j}}{(\alt_{j+1}^2/\alt_{j}^2;q)_{\nt_j}}.
\end{split}
\end{equation}
The proof follows from equations \eqref{5.7}, \eqref{5.8}, \eqref{5.9} and 
\eqref{5.10}.
\end{proof}

\subsection{The commutative algebra $\cAn$}\label{ss5.2}
We denote by $\cDza$ the associative subalgebra of $\cDz$ of difference 
operators with coefficients depending rationally on the parameters $\al_j$.
Clearly, the commutative algebra $\cAz$ defined in \prref{pr4.5} is 
contained in $\cDza$. Similarly, we denote by $\cDna$ the associative 
algebra of difference operators in the variables $n=(n_1,n_2,\dots,n_d)$ with 
coefficients depending rationally on $q^{n_1},q^{n_2},\dots,q^{n_d}$ and the 
parameters $\al_j$. We use $E_{n_k}$ to denote the forward shift in the 
variable $n_{k}$, i.e. for every function $f(n)=f(n_1,\dots,n_d)$ we have 
$E_{n_k} f(n)=f(n+e_{k})$.

Replacing $j$ by $d+1-k$ in \eqref{5.1a} we see that
$$q^{\nt_{d+1-k}}=\frac{\al_{k}z_{k}}{\al_{k+1}z_{k+1}}.$$
From this equation it follows that a forward $q$-shift in the variable $z_k$ 
will correspond to a forward shift in $\nt_{d+1-k}$ and a backward shift in 
$\nt_{d+2-k}$ for $k=2,3,\dots,d$ and to a forward shift in $\nt_{d}$ when 
$k=1$. Thus, in view of the duality established in \thref{th5.3}, we define 
a function $\bi$ as follows
\begin{subequations}\label{5.11}
\begin{align}
\bi(\al_0)&=\al_0\label{5.11a}\\
\bi(\al_j)&=\frac{\al_0\al_{d+1}\al_{d+2}}{\al_{d+2-j}}q^{1/2}\qquad\text{for }
j=1,2,\dots,d+1\label{5.11b}\\
\bi(\al_{d+2})&=\frac{\al_{1}}{\al_{0}}q^{-1/2}\label{5.11c}\\
\bi(z_{j})&=\frac{\al_{d+2-j}}{\al_0}q^{n_1+n_2+\cdots+n_{d+1-j}-1/2} 
\qquad\text{for }j=1,2,\dots,d \label{5.11d}\\
\bi (\E{z_j})&=E_{n_{d+1-j}}E_{n_{d+2-j}}^{-1},
\text{ for }j=1,2,\dots,d  \label{5.11e}
\end{align}
\end{subequations}
with the convention that $E_{n_{d+1}}$ is the identity operator. 

\begin{Lemma}\label{le5.4}
The mapping \eqref{5.11} extends to an isomorphism $\bi:\cDza\rightarrow\cDna$.
In particular, the operators $\fLn_k$ defined by 
\begin{equation}\label{5.12}
\fLn_k=\bi(\fLz_k) \quad \text{ for }k=1,2,\dots,d
\end{equation}
commute with each other and therefore
\begin{equation}\label{5.13}
\cAn=\bi(\cAz)=\C[\fLn_1,\fLn_2,\dots,\fLn_d]
\end{equation}
is a commutative subalgebra of $\cDna$.
\end{Lemma}

\begin{proof}
Using \eqref{5.11} one can check that
$$\bi(\E{z_k})\cdot \bi(f(z))=\bi(f(zq^{e_k}))\bi(\E{z_k}),$$
holds for every $k=1,2,\dots, d$, which shows that $\bi$ is a well defined 
homomorphism from $\cDza$ to $\cDna$. The fact that $\bi$ is one-to-one and 
onto follows easily.
\end{proof}

Recall that the eigenvalues $\mu_j$ of the operators $\fLz_j$ are given in 
\eqref{4.11b}. We denote
\begin{equation}\label{5.14}
\begin{split}
\kappa_j&=\bi^{-1}(\mu_j)=
-\left(1-\frac{\al_{d+1}\al_{d+2}}{\al_{d+1-j}z_{d+1-j}}\right)
\left(1-\frac{\al_{d+1}\al_{d+2}z_{d+1-j}}{\al_{d+1-j}}\right)\\
&=-1-\frac{\al_{d+1}^2\al_{d+2}^2}{\al_{d+1-j}^2}+
\frac{2\al_{d+1}\al_{d+2}}{\al_{d+1-j}}x_{d+1-j}.
\end{split}
\end{equation}

With these notations we can formulate the main result.
\begin{Theorem}\label{th5.5}
The polynomials $\Ph_d(n;x;\al)$ defined by equations
\eqref{2.9} and \eqref{5.4} diagonalize the algebras $\cAz$ and $\cAn$. 
More precisely, the following spectral equations hold
\begin{subequations}\label{5.15}
\begin{align}
&\fLz_j \Ph_d(n;x;\al) =\mu_j\Ph_d(n;x;\al)\label{5.15a}\\
&\fLn_j \Ph_d(n;x;\al) =\kappa_j\Ph_d(n;x;\al),\label{5.15b}
\end{align}
\end{subequations}
for $j=1,2,\dots,d$ where $\fLz_j,\mu_j$ are given in \eqref{4.11}
and $\fLn_j,\kappa_j$ are given in \eqref{5.12}, \eqref{5.14}.
\end{Theorem}
\begin{Remark}[Boundary conditions]\label{re5.6}
Since $\fLn_k$ contains backward shift operators, and since $\Ph_d(n;x;\al)$ 
is defined only for $n\in\N_0^d$ it is natural to ask what happens when 
$\fLn_k$ produces a term with a negative $n_j$ for some $j$. From 
\eqref{5.11e} it follows that $\bi(\E{z}^{\nu})$ with $\nu\in\{-1,0,1\}^d$
will contain a negative power of $E_{n_{d+2-j}}$ ($j\geq 2$) in one of the 
following two cases:\\
{\em Case 1: $\nu_j=1$, $\nu_{j-1}=0$.} In this case, $\bi(\E{z}^{\nu})$
will contain $E_{n_{d+2-j}}^{-1}$. Notice that the coefficient of 
$\E{z}^{\nu}$ is $C_{\nu}$ which has the factor 
$\left(1-\frac{\al_{j}z_{j}}{\al_{j-1}z_{j-1}}\right)$ 
(in $B_{j-1}^{0,1}$ - see formula 
\eqref{4.2b}). Since $\bi\left(1-\frac{\al_{j}z_{j}}{\al_{j-1}z_{j-1}}\right)
=(1-q^{-n_{d+2-j}})$ we see that this term is $0$ when $n_{d+2-j}=0$.\\
{\em Case 2: $\nu_j=1$, $\nu_{j-1}=-1$.} This time $\bi(\E{z}^{\nu})$ 
contains $E_{n_{d+2-j}}^{-2}$. The coefficient $C_{\nu}$ has the factor 
$B_{j-1}^{-1,1}= \left(1-\frac{\al_{j}z_{j}}{\al_{j-1}z_{j-1}}\right)
\left(1-\frac{q\al_{j}z_{j}}{\al_{j-1}z_{j-1}}\right)$
(see \eqref{4.2d} and \eqref{4.2e}). 
Since $\bi(B_{j-1}^{-1,1})=(1-q^{-n_{d+2-j}})(1-q^{1-n_{d+2-j}})$ 
we see that this coefficient is $0$ when $n_{d+2-j}=0$ or $1$.
\end{Remark}

\begin{proof}[Proof of \thref{th5.5}]
Equation \eqref{5.15a} follows immediately from \prref{pr4.5} and the fact 
that $\Ph_d(n;x;\al)$ and $P_d(n;x;\al)$ differ by a factor independent 
of $z$. It remains to prove \eqref{5.15b}.
Let us fix $n,\al$. By \leref{le5.1} for $z\in(\C^*)^d$  
we can find $(\nt,\zt,\alt)$ such that equations \eqref{5.1} hold. If 
$\nt\in\N_0^d$ then we can use \thref{th5.3} and replace $\Ph_d(n;x;\al)$
by $\Ph_d(\nt;\xt;\alt)$. Equation \eqref{5.15b} follows from the fact that 
the operator $\fLn_j$ in the variables $n$ with parameters $\al_j$ coincides 
with the operator $\fLz_j$ in the variables $\zt$ with parameters $\alt_j$.
We can use analytic continuation and complete the proof as follows. 
Fix $n,\al$ and write $z$ in terms of the dual 
variables $\nt$, i.e. we put
$$z_k=\frac{\al_{d+1}\al_{d+2}}{\al_k}
q^{\nt_{1}+\nt_{2}+\cdots+\nt_{d+1-k}}\text{ for }k=1,2,\dots,d.$$
Both sides of equation \eqref{5.15b} become Laurent polynomials in the 
variables $q^{\nt_1}$, $q^{\nt_2}$, $\dots$, $q^{\nt_d}$. Since \eqref{5.15b} 
is true for every $\nt\in\N_0^d$, we conclude that it must be true for 
arbitrary $\nt\in\C^d$, or equivalently, for arbitrary $z\in(\C^*)^d$.
\end{proof}

\section*{Acknowledgments}
I would like to thank Jeff Geronimo for stimulating discussions and valuable 
comments.

\end{document}